\documentclass[11pt]{article}

\usepackage[a4paper,margin=1in]{geometry}
\usepackage{amsmath,amssymb,amsthm,mathtools}
\usepackage{enumitem}
\usepackage{hyperref}
\usepackage{setspace}
\usepackage{cite}

\hypersetup{colorlinks=true,linkcolor=blue,urlcolor=blue}

\newcommand{\be}{\begin{equation} \label}
\newcommand{\ee}{\end{equation}}

\newtheorem{theorem}{Theorem}[section]
\newtheorem{proposition}[theorem]{Proposition}
\newtheorem{lemma}[theorem]{Lemma}

\newtheorem{remark}[theorem]{Remark}
\theoremstyle{definition}

\newcommand{\R}{\mathbb R}
\newcommand{\RN}{\mathbb R^N}
\newcommand{\Tmax}{T^*}
\newcommand{\eps}{\varepsilon}
\newcommand{\norm}[1]{\left\|#1\right\|}
\newcommand{\csch}{\operatorname{csch}}
\newcommand{\sech}{\operatorname{sech}}
\newcommand{\arsinh}{\operatorname{arsinh}}
\newcommand{\BUC}{\mathit{BUC}}

\title{\bf Strong gradient absorption \\ and initial decay prevent blow-up} 

\author{Philippe Souplet and Berikbol T. Torebek}

\date{}

\begin{document}
\maketitle

\begin{abstract}
 Consider positive classical solutions of the Cauchy problem 
for $u_t-\Delta u=u^p-\mu|\nabla u|^q$, with sufficiently smooth initial data $u_0$, in the strong absorption regime, 
namely $q\ge p>1$ and $\mu>0$.
It has remained a long standing open problem whether finite
time blow-up can occur if  $u_0$ decays at space infinity.
We here solve this problem, showing that the answer in negative.

To this end, we establish a novel, universal Bernstein-type gradient estimate, of independent interest,
valid only in the range $q\ge p$.
Moreover, the decay assumption on $u_0$ turns out to be essentially optimal.
Indeed we show that the solution blows up in finite time whenever $u_0$ 
has a positive limit at infinity,
and we determine the blow-up behavior of $u$ at space infinity under an additional assumption.
\end{abstract}



\section{Introduction}

\subsection{The problem}

 We consider nonnegative solutions of the parabolic problem
\be{P0}
\begin{cases}
 u_t-\Delta u=u^p-\mu|\nabla u|^q,& t>0,\ x\in\Omega,\\
u=0,& t>0,\ x\in\partial\Omega,\\
 u(0,x)=u_0(x),&x\in\Omega,
\end{cases}
\ee
where $p,q>1$, {$\mu>0$,}
\be{hypOmega}
\hbox{$\Omega$ is a, possibly unbounded, uniformly smooth domain of $\RN$, $N\ge 1$,}
\ee
 and 
$$u_0\in X_+=\big\{\phi\in BC^1(\overline\Omega);  \phi\ge 0,\ \phi_{|\partial\Omega}=0\big\}.$$
It is known that \eqref{P0} admits a unique, maximal classical solution $u\ge 0$, 
of existence time denoted by $\Tmax=\Tmax(u_0)\in (0,\infty]$
(see section~\ref{LocTheory} below for detailed statements).

Problem \eqref{P0} can be seen as a compound 
of two very classical equations: the Fujita equation from blow-up theory
\be{FujitaEq}
u_t-\Delta u=u^p
\ee
(see \cite[Chapter~II]{QuittnerSouplet2019} for detailed references),
and the diffusive Hamiton-Jacobi equation from stochastic 
optimal control
\be{DHJEq}
u_t-\Delta u+|\nabla u|^q=0
\ee
(see \cite{PLL82,BCD97,FMS06} and \cite[Section~40]{QuittnerSouplet2019}).
Problem \eqref{P0} is significant for the qualitative theory of nonlinear diffusion equations,
as a model for the competition between the fundamental mechanisms 
featured in \eqref{FujitaEq} and \eqref{DHJEq}: a superlinear source and a nonlinear gradient absorption.
Starting with the work \cite{ChipotWeissler1989}, it has stimulated a large amount of research,
especially on such issues as blow-up, global existence, steady states,
critical exponents, description of asymptotic behavior, and analysis of singularities.
A survey of the earlier work can be found in \cite{Souplet01} and in \cite[Sections~36--39]{QuittnerSouplet2019}.
For more recent contributions
we refer, in the parabolic setting, to \cite{Bou25, EZ11, GNZ17, TZ19} (finite time blow-up profiles),
 \cite{FL20, JSP20, Tor23} (Fujita type phenomena),
 \cite{ST07} (global decaying solutions),
 \cite{ZZ19} (free boundary problems),
  \cite{HK15,Ng17,RHKK25} (numerical aspects of blow-up)
 and, concerning steady states, to
\cite{ABGQ16,AGQ13,GQS19} (nonexistence results) and
\cite{BVV23a,BVV23b,BVGV20} (a priori estimates, singularities and boundary traces).

Our main focus here is on the strong absorption case $q\ge p$.
 It is known that $q=p$ plays the role of a critical exponent for blow-up vs.~global existence in {\it bounded} domains.
Namely, when $\Omega$ is bounded, it was proved in \cite{Fila91} that all solutions are global and bounded if $q\ge p$,
whereas, as shown in \cite{SoupletWeissler1997},  finite time blow-up occurs for suitably large initial data if $q<p$
(and the latter conclusion actually holds for any $\Omega$).
Interestingly, the situation is different in the case $\Omega=\R^N$, i.e. for
\begin{equation}
\label{P}
\begin{cases}
 u_t-\Delta u=u^p-\mu|\nabla u|^q,
& t>0,\ x\in\R^N,\\
 u(0,x)=u_0(x), &x\in\R^N.
\end{cases}
\end{equation}
Indeed, when $q\ge p$, the existence of {\it global unbounded} solutions of \eqref{P} for
sufficiently large initial with compact support was established in \cite{SoupletWeissler1999}.
On the other hand, whatever the value of $q\ge 1$, \eqref{P} always has special finite time blow-up solutions,
namely the flat solutions, corresponding to $u_0(x)=c>0$
(since the latter solve the ODE $y'=y^p$).
However, {in spite of the substantial activity on this problem}, the following has remained a long standing open question:
\be{OP}
\hbox{Can finite time blow-up for problem \eqref{P} occur if $q\ge p$ and $u_0$ decays at space infinity~?}
\ee
{See \cite{Souplet01} and \cite[Open Problem~4.1 p.~658]{QuittnerSouplet2019}.}
The first goal of this paper 
is solve this question, and it will turn out that the answer in negative.
This will be a consequence of a new, {\it universal gradient estimate} for \eqref{P},
of independent interest,
valid only in the range $q\ge p$.

\subsection{Main results}

Throughout,  we set
$$C_0(\Omega)=\Big\{\phi\in C(\overline\Omega);\  \phi_{|\partial\Omega}=0\ \ \hbox{and}\ \lim_{|x|\to\infty} \phi(x)=0\Big\}.$$
Our first main result completely solves the open problem \eqref{OP}.
Moreover the result holds for any domain.

\begin{theorem}[Global existence]
\label{thm:main-BUC}
Assume \eqref{hypOmega} and
\be{range-main}
q\ge p>1,\qquad\mu>0.
\ee
Then for any $u_0\in {X_+\cap C_0(\Omega})$,
the maximal solution of \eqref{P0} is global i.e., $\Tmax=\infty$.
\end{theorem}

Our second main result is a 
universal gradient estimate in the whole range $q\ge p$ for $\Omega=\R^N$,
and it is the main ingredient in the proof of Theorem~\ref{thm:main-BUC}.

\begin{theorem}[Universal gradient estimate]
\label{thm:universal}
Assume \eqref{range-main} and let $0<\alpha<\frac12$.
For all $u_0\in {X_+}$, the maximal solution of \eqref{P} satisfies
\begin{equation}
\label{grad-u}
 |\nabla u|^2
 \le L(t)\bigl(u^{2-2\alpha}+u^2\bigr),\quad 0<t<\Tmax,\ x\in\RN,
\end{equation}
with $L(t)=B\bigl(1+t^{-1/\sigma}\bigr)$, where $\sigma=\min(1,\frac q2)$ and the constant $B>0$ depends only on $p,q,\mu,\alpha$.
\end{theorem}

Our last main result shows that the hypothesis $u_0\in C_0(\Omega)$
in Theorem~\ref{thm:main-BUC} is essentially sharp.
Namely consider the case when $\Omega=\R^N$ and 
\begin{equation}
\label{u0-limit}
 \lim_{|x|\to\infty}u_0(x)=\ell>0.
\end{equation}
We show that $u$ then blows up in finite time, and we determine the blow-up profile at space infinity under
the additional assumption that $u_0\le \ell$.
We stress that this actually holds for all $q>1$.
Here, for $s>0$, we denote
by $T_s:=\frac{s^{1-p}}{p-1}$,
the blow-up time of the solution of the corresponding ODE: $\zeta'=\zeta^p$ with $\zeta(0)=s$.

\begin{theorem}
\label{thm:asymptotically-constant}
Assume $p,q>1$, $\mu\ge 0$ and let $u_0\in {X_+}$ 
satisfy \eqref{u0-limit}.
Let $u$ be the maximal solution of \eqref{P}.
\begin{enumerate}[label=\textup{(\alph*)},leftmargin=*]
\item[(i)]
Then {$\Tmax<\infty$} and, more precisely,
$T_{\norm{u_0}_\infty}\le\Tmax\le T_\ell$.

\item[(ii)]
The solution $u$ satisfies
\be{evolving-tail}
\lim_{|x|\to\infty}u(t,x)
 =\bigl(\ell^{1-p}-(p-1)t\bigr)^{-1/(p-1)},
 \qquad 0\le t<\Tmax.
\ee
In particular, if $u_0\le \ell$,
then $\Tmax=T_\ell$ and the right hand side of 
\eqref{evolving-tail} blows up as $t\to \Tmax$.
\end{enumerate}
\end{theorem}

\subsection{Remarks and discussion}

\begin{remark}
 \emph{(Bounded vs.~unbounded global solutions)}  
Under assumption \eqref{range-main}, it was shown
 in \cite{SoupletWeissler1999,Souplet99} (see also \cite[Theorem~36.4]{QuittnerSouplet2019})
that all solutions of problem \eqref{P0} with $u_0\in X_+\cap C_0(\overline\Omega)$
are global and bounded 
 if and only if $\Omega$ has finite inradius, i.e.
$$\rho(\Omega):=\sup\big\{R>0;\ \hbox{$\Omega$ contains a ball of radius $R$}\big\}<\infty.$$
In particular, in the case $\rho(\Omega)=\infty$,  \cite[Theorem~2]{SoupletWeissler1999}
and \cite[Theorem~36.4]{QuittnerSouplet2019}
guarantee the existence of $u_0\in X_+\cap C_0(\overline\Omega)$ such that
$u$ becomes unbounded either in infinite or finite time.
Theorem~\ref{thm:main-BUC} rules out the latter possibility.
\end{remark}

\begin{remark}
\label{rem:comments}
 \emph{(Specificity and meaning of \eqref{grad-u})}  
\begin{enumerate}[label=\textup{(\alph*)},leftmargin=*]
\item Estimate  \eqref{grad-u} is valid {\bf only} in the range $q\ge p$.
 In fact for $p>q$, by \cite{SoupletWeissler1997}, finite time blow-up occurs for all sufficiently large smooth $u_0$ with compact support.
 For any such solution and any $\alpha\in(0,1/2]$, it follows from Lemma~\ref{prop:tails2}  below that
$$
 \sup_{(\tau,\Tmax)\times{\R}^N} \,\frac{|\nabla u|^2}{u^2+u^{2-2\alpha}}=\infty,
 \quad 0<\tau<\Tmax,
$$
so that even a weaker (nonuniversal) estimate of the form \eqref{grad-u} cannot hold in this range.

\item The bound \eqref{grad-u} does not involve $u_0$ in any way and does not involve $N$ either.
It can be recast under the form
\begin{equation}
\label{universal}
  \bigl|\nabla\arsinh(u^\alpha)(t,x)\bigr|^2
 \le \alpha^2 B\bigl(1+t^{-1/\sigma}\bigr).
 \end{equation}
Writing $Z=\arsinh(u^\alpha)$ and using
$|\nabla Z|=\frac{|\nabla u^\alpha|}{\sqrt{1+u^{2\alpha}}}$,
estimate \eqref{universal} gives, for $t\ge\tau>0$,
\[
 |\nabla u^\alpha|^2\le2\alpha^2L(\tau)
 \quad\text{on }\{u\le1\},
 \qquad
 |\nabla\log u|^2\le2L(\tau)
 \quad\text{on }\{u\ge1\}.
\]
The transform is thus power-like where the solution is small and logarithmic where it is large, and the estimate is a Lipschitz bound in the corresponding variable.
 We refer to Remark~\ref{rem:where} for an explanation of the role of the hypotheses on the parameters
in the proof and applications of estimate \eqref{grad-u}.
\end{enumerate}
\end{remark}

\begin{remark} \label{AsymptConst}
 \emph{(Asymptotically constant initial data)}
\begin{enumerate}[label=\textup{(\alph*)},leftmargin=*]
\item
For asymptotically constant data $u_0\in {X_+}$, i.e.~$\lim_{|x|\to\infty}u_0(x)=\ell\ge 0$, the combination of 
Theorems~\ref{thm:main-BUC} and \ref{thm:asymptotically-constant} 
leads to the dichotomy $\Tmax=\infty\Leftrightarrow\ell=0$.

\item By \eqref{evolving-tail}, the state of the solution at spatial infinity evolves exactly according to the scalar equation $\zeta'=\zeta^p$: the gradient absorption has no effect there.
We refer to \cite{GU06,Shi08} for related results on blow-up at space infinity
with asymptotically constant initial data for equations of the form
$u_t-\Delta u=f(u)$.

\item Assertion (i) of Theorem~\ref{thm:asymptotically-constant} remains true if assumption \eqref{u0-limit} is weakened to 
$\ell:=\liminf_{|x|\to\infty}u_0(x)>0$.
See Remark~\ref{JustifClaim}. 

\item The assumption $\mu\ge 0$ in Theorem~\ref{thm:asymptotically-constant} is only used for the proof of $T_{\norm{u_0}_\infty}\le\Tmax$
and of the second part of assertion (ii). The other conclusions are valid also for $\mu<0$.
\end{enumerate}
\end{remark}

\begin{remark}[A model barrier on bounded domains]
\label{rem:barrier}

The condition $q\ge p$ makes the absorption dominate the source wherever the solution is steep relative to its size.  A simple way to see this is the following barrier.  For
$k:=\mu^{-1/q}$ and $M\ge1$,
the function $U(t,x)=M\exp(k^2t+kx_1)$ satisfies
$U_t=\Delta U=k^2U$ and $|\nabla U|=kU$,
hence
\[
 U_t-\Delta U-U^p+\mu|\nabla U|^q
 =U^q-U^p\ge0
 \qquad\text{wherever }U\ge1 .
\]
After a translation in $x_1$, one has $U\ge1$ on any given bounded domain, where this barrier dominates the solution of the homogeneous Dirichlet problem for all times and rules out blow-up;  related mechanisms
are used in \cite{Fila91} and \cite{SoupletWeissler1999}. In $\RN$ no translate of $U$ stays above $1$, since $U\to0$ as $x_1\to-\infty$.  The difficulty is therefore to control the solution uniformly near spatial infinity, and we achieve this through a gradient estimate that does not depend on the initial datum.
\end{remark}

\subsection{Strategy}

Theorem~\ref{thm:universal} is proved by the Bernstein technique applied to the hybrid transform $Z=\arsinh(u^\alpha)$.  The equation for $W=|\nabla Z|^2$ contains a coercive term $-(D(Z)+S'(Z))W^{1+\sigma}$, in which the weight $D$ comes from the diffusion and controls small $u$, while the weight $S'$ comes from the gradient absorption and controls large $u$.  For $W$ large this term dominates the contribution $R'(Z)W$ of the source, \emph{uniformly in $Z$}; this uniformity is where $q\ge p$ enters, and it is what makes the resulting bound independent of the datum.  Comparison with a spatially homogeneous barrier decreasing from $+\infty$ at $t=0$ then gives \eqref{universal}.  The singularity of the transform at $u=0$ is removed by a positive-floor approximation.

 As for the proof of Theorem~\ref{thm:main-BUC}, the idea is as follows.
 In the case $\Omega=\RN$, for $t\ge\tau>0$, 
estimate \eqref{grad-u} bounds $|\nabla u|^2$ by powers of $u$ with exponents at least one.
 From this one can show that $V=1-e^{-u}$ satisfies $V_t-\Delta V\le CV$, and its spatial tails are controlled by the heat semigroup, uniformly on bounded time intervals.  So $u\le1$ outside a fixed ball up to the putative blow-up time.  Inside the ball, the Lipschitz bound \eqref{universal} converts this into an $L^\infty$ bound.
 For general $\Omega$, the same $L^\infty$ bound follows from the case $\Omega=\RN$ and a comparison argument. Since, by the known blow-up alternative for $\mu>0$, blow-up has to occur in $L^\infty$ norm,
we reach a contradiction.

Finally, the proof of Theorem~\ref{thm:asymptotically-constant}  is based on the preservation of spatial limits,
obtained by semigroup estimates and a fixed point argument on suitable function spaces,
and on comparison arguments.

\subsection{Organization}

Section~\ref{sec:universal} introduces the hybrid transform and proves Theorem~\ref{thm:universal}.  
 Theorems~\ref{thm:main-BUC} and \ref{thm:asymptotically-constant}
are proved in Sections~\ref{sec:global} and \ref{sec:dichotomy}, respectively.
A number of known facts concerning the 
heat semigroup, the local theory and comparison principles
 are collected in Appendix~\ref{sec:prelim}.

\goodbreak

\section{The hybrid transform and the universal estimate}
\label{sec:universal}

Throughout this section we assume \eqref{range-main} and fix
\begin{equation*}
\label{alpha}
 0<\alpha<\textstyle\frac12 .
\end{equation*}
The power transform $u^{(q-p)/q}$, which is natural when $q>p$, degenerates on the diagonal $q=p$.  We therefore use a transform which is power-like for small $u$ and logarithmic for large $u$, and which requires no relation between $\alpha$ and the exponents:
\begin{equation}
\label{yZ}
 y:=u^\alpha,
 \qquad
 Z:=\arsinh y=\arsinh(u^\alpha),
 \qquad
 W:=|\nabla Z|^2 .
\end{equation}
Set
\begin{equation*}
\label{parameters}
 A:=\frac{1-\alpha}{\alpha}>1,
 \qquad
 \gamma:=1+\frac{p-1}{\alpha},
 \qquad
 \delta:=\frac{(q-1)(1-\alpha)}{\alpha}>0,
\end{equation*}
and, for $Z>0$,
$$
\begin{aligned}
\label{coeff-a}
 a(Z)&:=\tanh Z+A\coth Z,\\
 R(Z)&:=\alpha\frac{(\sinh Z)^\gamma}{\cosh Z},\\
 S(Z)&:=\mu\alpha^{1-q}
 (\sinh Z)^\delta(\cosh Z)^{q-1},\\
 D(Z)&:=-a'(Z)=A\csch^2Z-\sech^2Z .
\end{aligned}
$$
For $q=p$ these reduce to the exponents of the diagonal case; the point of this section is that no case distinction is needed.

\subsection{Equation for the hybrid transform}

\begin{lemma}[Equation for $Z=\arsinh(u^\alpha)$]
\label{lem:transform}
Let $u$ be a positive classical solution of $u_t-\Delta u=u^p-\mu|\nabla u|^q$ in an open set $Q\subset(0,\infty)\times\RN$, and let $y,Z,W$ be given by \eqref{yZ}.  Then, in $Q$,
\begin{align}
\label{yeq}
 y_t-\Delta y
 &=A\frac{|\nabla y|^2}{y}
 +\alpha y^\gamma
 -\mu\alpha^{1-q}y^\delta|\nabla y|^q,\\
\label{Zeq}
 Z_t-\Delta Z
 &=a(Z)W+R(Z)-S(Z)W^{q/2}.
\end{align}
\end{lemma}

\begin{proof}
From $y=u^\alpha$,
\[
 y_t-\Delta y
 =\alpha u^{\alpha-1}(u_t-\Delta u)
 +\alpha(1-\alpha)u^{\alpha-2}|\nabla u|^2,
 \qquad
 \nabla u=\alpha^{-1}u^{1-\alpha}\nabla y .
\]
Hence:
\begin{itemize}[leftmargin=*]
\item $\alpha(1-\alpha)u^{\alpha-2}|\nabla u|^2
 =A\,u^{-\alpha}|\nabla y|^2=A|\nabla y|^2/y$;
\item $\alpha u^{\alpha+p-1}
 =\alpha y^{(\alpha+p-1)/\alpha}=\alpha y^\gamma$;
\item $\mu\alpha u^{\alpha-1}|\nabla u|^q
 =\mu\alpha^{1-q}u^{(q-1)(1-\alpha)}|\nabla y|^q
 =\mu\alpha^{1-q}y^\delta|\nabla y|^q$,
\end{itemize}
which is \eqref{yeq}.  Next, $y=\sinh Z$ gives
\[
 \nabla y=(\cosh Z)\,\nabla Z,
 \qquad
 y_t-\Delta y=(Z_t-\Delta Z)\cosh Z-W\sinh Z .
\]
Substituting into \eqref{yeq}, using
\[
 \frac{|\nabla y|^2}{y}=\frac{\cosh^2Z}{\sinh Z}\,W,
 \qquad
 |\nabla y|^q=(\cosh Z)^qW^{q/2},
\]
and dividing by $\cosh Z$ gives \eqref{Zeq}.
\end{proof}

\subsection{Structure of the coefficients}

The following lemma is the only place where $q\ge p$ is used.

\begin{lemma}[Coercive coefficient structure]
\label{lem:coeff}
The functions $D$, $R'$ and $S'$ are positive on $(0,\infty)$.  Moreover there exist constants $K>0$ and $g_0>0$, depending only on $p,q,\mu,\alpha$, such that
\begin{equation}
\label{coeff-bounds}
 R'(Z)\le K\bigl(D(Z)+S'(Z)\bigr),
 \qquad
 D(Z)+S'(Z)\ge g_0,
 \qquad Z>0 .
\end{equation}
\end{lemma}

\begin{proof}
\emph{Positivity.}  Since $A>1$,
\[
 D(Z)=\frac{A\cosh^2Z-\sinh^2Z}{\sinh^2Z\cosh^2Z}>0 .
\]
Next,
\begin{equation*}
\label{Rprime}
 R'(Z)
 =\alpha(\sinh Z)^{\gamma-1}
 \frac{\gamma+(\gamma-1)\sinh^2Z}{\cosh^2Z},
\end{equation*}
so that, using $\gamma+(\gamma-1)\sinh^2Z\le\gamma\cosh^2Z$,
\begin{equation}
\label{Rprime-bound}
 0<R'(Z)\le\alpha\gamma(\sinh Z)^{\gamma-1}.
\end{equation}
Finally
\begin{equation}
\label{Sprime}
 S'(Z)=S(Z)\bigl(\delta\coth Z+(q-1)\tanh Z\bigr)>0 .
\end{equation}

\emph{Behavior as $Z\downarrow0$.}  We have $D(Z)\ge A\csch^2Z-1\to\infty$, while $R'(Z)\to0$ by \eqref{Rprime-bound}, since $\gamma>1$.  Hence $D+S'\to\infty$ and $R'/(D+S')\to0$.

\emph{Behavior for $Z\ge1$.}  Using $\tanh Z\ge\tanh(1)$ and $\cosh Z\ge\sinh Z$ in \eqref{Sprime},
\[
 S'(Z)\ge(q-1)\tanh(1)\,\mu\alpha^{1-q}(\sinh Z)^{\delta+q-1},
 \qquad
 \delta+q-1=\frac{q-1}{\alpha}.
\]
In particular $S'(Z)\to\infty$ as $Z\to\infty$.  Since $\gamma-1=(p-1)/\alpha$, this and \eqref{Rprime-bound} give, for $Z\ge1$,
\begin{equation}
\label{ratio-bound}
 \frac{R'(Z)}{S'(Z)}
 \le\frac{\alpha^q\gamma}{(q-1)\tanh(1)\,\mu}
 (\sinh Z)^{-(q-p)/\alpha}
 \le\frac{\alpha^q\gamma}{(q-1)\tanh(1)\,\mu}
 (\sinh(1))^{-(q-p)/\alpha}.
\end{equation}
Here we used $q\ge p$; this is the only place where the hypothesis enters.

\emph{Conclusion.}  The function $R'/(D+S')$ is continuous on $(0,\infty)$, tends to $0$ at the origin, and is bounded on $[1,\infty)$ by \eqref{ratio-bound}; hence it is bounded, which is the first inequality in \eqref{coeff-bounds}.  The function $D+S'$ is continuous and positive and tends to $+\infty$ at both ends of $(0,\infty)$; hence it has a positive minimum, which is the second.
\end{proof}

\begin{remark}
\label{rem:ratio}
A finer computation identifies the behavior at infinity.  Using $\sinh Z\sim\cosh Z\sim e^Z/2$, we get
\[
 R'(Z)\sim\alpha(\gamma-1)2^{1-\gamma}e^{(\gamma-1)Z},
 \qquad
 S'(Z)\sim\mu\alpha^{1-q}(\delta+q-1)
 2^{-(\delta+q-1)}e^{(\delta+q-1)Z},
\]
with $\gamma-1=(p-1)/\alpha$ and $\delta+q-1=(q-1)/\alpha$.  Hence
\begin{equation*}
\label{ratio-infty}
 \frac{R'(Z)}{S'(Z)}
 \longrightarrow
 \begin{cases}
 \dfrac{\alpha^p}{\mu}, & q=p,\\[2mm]
 0, & q>p,
 \end{cases}
 \qquad Z\to\infty .
\end{equation*}
If $q<p$ the same ratio tends to $+\infty$, and the argument breaks down: the source then dominates the absorption at large amplitudes, and no bound uniform in $Z$ is available.
\end{remark}

\subsection{The Bernstein identity}

\begin{lemma}[Equation for $W$]
\label{lem:bernstein}
In the setting of Lemma~\ref{lem:transform}, assume in addition that $\nabla u\in C^{2,1}(Q)$.
Then $W\in C^{2,1}(Q)$ and, in $Q$,
\begin{equation}
\label{Weq}
 W_t-\Delta W-b\cdot\nabla W
 =-2D(Z)W^2+2R'(Z)W-2S'(Z)W^{1+q/2}-2|D^2Z|^2,
\end{equation}
where
\begin{equation}
\label{drift}
 b:=\bigl(2a(Z)-q\,S(Z)|\nabla Z|^{q-2}\bigr)\nabla Z,
\end{equation}
with the convention $|\nabla Z|^{q-2}\nabla Z:=0$ where $\nabla Z=0$.
\end{lemma}

\begin{proof}
Since $Z=F(u)$ with $F(s)=\arsinh(s^\alpha)$ smooth on $(0,\infty)$, the assumed regularity of $u$ gives $Z\in C^3$ in $x$ with $\nabla Z_t\in C(Q)$, hence $W\in C^{2,1}(Q)$.  The Bochner identity
\[
 \Delta|\nabla Z|^2=2|D^2Z|^2+2\nabla Z\cdot\nabla\Delta Z,
\]
which is the starting point of the Bernstein method (see \cite{Serrin1971,LSU1968,BartierSouplet2004}), gives
\[
 W_t-\Delta W=2\nabla Z\cdot\nabla(Z_t-\Delta Z)-2|D^2Z|^2 .
\]
We insert \eqref{Zeq} and use $\nabla W=2D^2Z\,\nabla Z$.  Since $\xi\mapsto|\xi|^q$ is $C^1$ with gradient $q|\xi|^{q-2}\xi$, we have $\nabla(|\nabla Z|^q)=\frac q2|\nabla Z|^{q-2}\nabla W$ everywhere, the right-hand side being interpreted as $0$ where $\nabla Z=0$.  Therefore
\begin{align*}
 2\nabla Z\cdot\nabla\bigl(a(Z)W\bigr)
 &=2a'(Z)W^2+2a(Z)\nabla Z\cdot\nabla W,\\
 2\nabla Z\cdot\nabla R(Z)
 &=2R'(Z)W,\\
 -2\nabla Z\cdot\nabla\bigl(S(Z)|\nabla Z|^q\bigr)
 &=-2S'(Z)W^{1+q/2}
 -q\,S(Z)|\nabla Z|^{q-2}\nabla Z\cdot\nabla W .
\end{align*}
Collecting the first-order terms and using $a'=-D$ gives \eqref{Weq}.
\end{proof}

Note that,
when $q<2$ the factor $|\nabla Z|^{q-2}\nabla Z$ in \eqref{drift} is understood by continuous extension at $\nabla Z=0$; indeed $\bigl||\nabla Z|^{q-2}\nabla Z\bigr|=W^{(q-1)/2}\to0$ as $W\downarrow0$, because $q>1$.

\subsection{Proof of Theorem~\ref{thm:universal}}

Since $u_0\not\equiv0$, Proposition~\ref{prop:local}(i) gives $u>0$ in $(0,\Tmax)\times\RN$.  Let $K$ and $g_0$ be the constants of Lemma~\ref{lem:coeff}, and set
\begin{equation}
\label{B0}
 \sigma:=\min\Bigl\{1,\frac q2\Bigr\},
 \qquad
  B_0:=\max\bigl\{1,(2K)^{1/\sigma}\bigr\}.
\end{equation}

\medskip
\noindent\emph{Step 1: a differential inequality.}
Let $v$ be a positive solution of \eqref{P} in an open set $Q$ with $\nabla v\in C^{2,1}(Q)$, and let $Z,W,b$ be as above.  We claim that
\begin{equation}
\label{W-ineq}
 W_t-\Delta W-b\cdot\nabla W
 \le-g_0W^{1+\sigma}
 \qquad\text{at every point of $Q$ where }W\ge B_0.
\end{equation}
First, since $W\ge B_0\ge1$ and $D,S'>0$,
\begin{equation*}
\label{power-combine}
 D(Z)W^2+S'(Z)W^{1+q/2}
 \ge\bigl(D(Z)+S'(Z)\bigr)W^{1+\sigma}.
\end{equation*}
Indeed, if $q\ge2$ then $\sigma=1$ and $W^{1+q/2}\ge W^2$; if $q<2$ then $\sigma=q/2$ and $W^2\ge W^{1+q/2}$.  Dropping the term $-2|D^2Z|^2\le0$ in \eqref{Weq} and using \eqref{coeff-bounds} together with $2K\le B_0^\sigma\le W^\sigma$,
\[
 W_t-\Delta W-b\cdot\nabla W
 \le-2(D+S')W^{1+\sigma}+2K(D+S')W
 \le-(D+S')W^{1+\sigma}
 \le-g_0W^{1+\sigma}.
\]

\medskip
\noindent\emph{Step 2: barrier.}
Let 
\[
 \psi(t):=\Bigl(\frac{\sigma g_0}{2}t\Bigr)^{-1/\sigma},
 \qquad t>0,
\]
so that $\psi'=-\frac{g_0}{2}\psi^{1+\sigma}$,
 and set
$$h(t):=B_0+\psi(t)\le B\bigl(1+t^{-1/\sigma}\bigr)=:L(t)$$
where, in view of Lemma~\ref{lem:coeff} and \eqref{B0},
$B>0$ depends only on $p,q,\mu,\alpha$.
  Since $h\ge\psi$ and $h\ge B_0$,
\begin{equation}
\label{barrier}
 -g_0h^{1+\sigma}
 =-\frac{g_0}{2}h^{1+\sigma}-\frac{g_0}{2}h^{1+\sigma}
 \le-\frac{g_0}{2}\psi^{1+\sigma}-\frac{g_0}{2}B_0^{1+\sigma}
 =h'-c,
 \qquad
 c:=\frac{g_0}{2}B_0^{1+\sigma}>0 .
\end{equation}
Note that $h$ is decreasing, with $h(t)\to\infty$ as $t\downarrow0$.

\medskip
\noindent\emph{Step 3: positive-floor approximation.}
Fix $T\in(0,\Tmax)$, and let $\eta_0$ and $u_\eta$, $0<\eta\le\eta_0$, be as in Proposition~\ref{cor:consequences}.  Set
\[
 Z_\eta:=\arsinh(u_\eta^\alpha),
 \qquad
 W_\eta:=|\nabla Z_\eta|^2,
\]
and let $b_\eta$ be given by \eqref{drift} with $Z_\eta$ in place of $Z$.  On $[0,T]\times\RN$ we have $\eta\le u_\eta\le M_\eta$ and $|\nabla u_\eta|\le G_\eta$ for some finite $M_\eta,G_\eta$.  Therefore:
\begin{itemize}[leftmargin=*]
\item $Z_\eta$ takes values in the compact interval $[\arsinh\eta^\alpha,\arsinh M_\eta^\alpha]\subset(0,\infty)$;
\item $W_\eta\le\alpha^2\eta^{2\alpha-2}G_\eta^2$, since $|\nabla Z_\eta|\le\alpha u_\eta^{\alpha-1}|\nabla u_\eta|$;
\item $b_\eta$ is bounded on $(0,T]\times\RN$, since
$|b_\eta|\le2|a(Z_\eta)||\nabla Z_\eta|+qS(Z_\eta)|\nabla Z_\eta|^{q-1}$.
\end{itemize}
By \eqref{HigherReg} in Proposition~\ref{prop:local}(ii), $\nabla u_\eta\in C^{2,1}((0,T]\times\RN)$, so Lemma~\ref{lem:bernstein} and Step~1 apply to $u_\eta$ on $(0,T]\times\RN$, and $W_\eta\in C^{2,1}((0,T]\times\RN)$.

Since $W_\eta$ is bounded and $h(s)\to\infty$ as $s\downarrow0$, we may choose $s\in(0,T)$ with $h(s)\ge\sup W_\eta$.  On $[s,T]$ we apply Lemma~\ref{lem:MP}: at any point of $(s,T]\times\RN$ where $W_\eta>h(t)$ we have $W_\eta>B$, so \eqref{W-ineq} and \eqref{barrier} give
\[
 (W_\eta)_t-\Delta W_\eta-b_\eta\cdot\nabla W_\eta
 \le-g_0W_\eta^{1+\sigma}
 <-g_0h^{1+\sigma}
 \le h'-c .
\]
Lemma~\ref{lem:MP} yields $W_\eta\le h$ on $[s,T]\times\RN$.  Since $h$ is decreasing and $W_\eta\le h(s)$ on $(0,s]$, in fact
\begin{equation}
\label{Weta}
 W_\eta(t,x)\le L(t),
 \qquad 0<t\le T,\quad x\in\RN,
 \quad 0<\eta\le\eta_0 .
\end{equation}
In particular, no limit $s\downarrow0$ is needed: the bound on $(0,s]$ is automatic.

\medskip
\noindent\emph{Step 4: passage to the limit.}
Fix $(t,x)\in(0,T]\times\RN$.  Then $u_\eta(t,x)\to u(t,x)>0$ and $\nabla u_\eta(t,x)\to\nabla u(t,x)$.  Moreover $\nabla Z_\eta=\Phi(u_\eta,\nabla u_\eta)$ with
\[
 \Phi(s,\xi):=\frac{\alpha s^{\alpha-1}\xi}{\sqrt{1+s^{2\alpha}}}
\]
continuous on $(0,\infty)\times\RN$.  Hence $W_\eta(t,x)\to W(t,x)$, and \eqref{Weta} gives \eqref{universal},
 hence \eqref{grad-u}, on $(0,T]\times\RN$.  Since $T<\Tmax$ is arbitrary,
 the theorem follows. 
\qed

\begin{remark}
\label{rem:where}
\emph{(Where the hypotheses enter)}  In the proof of Theorem~\ref{thm:universal} the condition $q\ge p$ is used only once, in Lemma~\ref{lem:coeff}, to compare the growth of the reaction and absorption coefficients of the transformed equation as $u\to\infty$; it is used again, independently, in the barrier computations of Remarks~\ref{rem:barrier} and \ref{rem:localization}.  The restriction $\alpha<\frac12$ is used twice: in Lemma~\ref{lem:coeff}, where it gives $A>1$ and hence $D>0$ and,
when applying Theorem~\ref{thm:universal}, 
in Lemma~\ref{prop:tails}, where it gives $2-2\alpha>1$.
\end{remark}

\section{Global existence}
\label{sec:global}

The proof is split in several subsections, displaying independently the different arguments involved, 
in order to clearly emphasize the role of the hypotheses.
The assumption $q\ge p$ enters in the proof at the final step only, through the use of the gradient estimate from 
Theorem~\ref{thm:universal}. 

\subsection{Two elementary lemmas}

\begin{lemma}[Uniform heat tails]
\label{lem:heat-tail}
If $g\in C_0(\RN)$ and $T>0$, then
\begin{equation*}
\label{heat-tail}
 \lim_{R\to\infty}
 \bigg\{ \sup_{0\le t\le T}\sup_{|x|\ge R}
 |e^{t\Delta}g(x)|\bigg\}=0. 
\end{equation*}
\end{lemma}

\begin{proof}
The map $[0,T]\ni t\mapsto e^{t\Delta}g\in C_0(\RN)$ is continuous \cite{Pazy1983}, so its image is compact.  A compact subset of $C_0(\RN)$ admits a finite $\eps$-net for every $\eps>0$, and therefore vanishes uniformly at infinity.
\end{proof}

\begin{lemma}
\label{lem:polyexp}
For every $r\ge1$,
\begin{equation*}
\label{polyexp}
 e^{-s}s^r\le K_r(1-e^{-s}),
 \qquad s\ge0,
 \qquad
 \hbox{ where }\ K_r:=\sup_{s>0}\frac{s^r}{e^s-1}<\infty .
\end{equation*}
\end{lemma}

\begin{proof}
The inequality is equivalent to $s^r/(e^s-1)\le K_r$.  As $s\downarrow0$ the quotient behaves like $s^{r-1}$, which is bounded because $r\ge1$; it tends to zero exponentially as $s\to\infty$.
\end{proof}

\subsection{Uniform tails}

\begin{lemma}[Finite-time tail preservation]
\label{prop:tails}
Assume $\Omega=\R^N$, $p,q>1$, $\mu\ge 0$. Let $u_0\in X_+\cap C_0(\RN)$ 
be such that $\Tmax<\infty$ and let $\tau\in(0,\Tmax)$.
If
\be{hyp-tails}
\sup_{(\tau,\Tmax)\times{\R}^N} \, \frac{|\nabla u|^2}{u^a+u^b}<\infty
\quad\hbox{ for some $a,b\ge 1$,}
\ee
then 
\begin{equation}
\label{tails}
 \lim_{R\to\infty}
\bigg\{ \sup_{\tau\le t<\Tmax\atop |x|\ge R}u(t,x)\Big\}=0 .
\end{equation}
\end{lemma}

We stress that this lemma, as well as that in the next subsection, does not require  $q\ge p$.
Also the ratio in  \eqref{hyp-tails} is well defined since $u(t,\cdot)>0$ for each $t\in(0,\Tmax)$, owing to $u_0\not\equiv0$.

\begin{proof}
Let $V:=1-e^{-u}\in[0,1)$.  Then $V\in C^{2,1}((0,\Tmax)\times\RN)$ and
\[
 V_t-\Delta V
 =e^{-u}\bigl(u_t-\Delta u+|\nabla u|^2\bigr)
 =e^{-u}\bigl(u^p-\mu|\nabla u|^q+|\nabla u|^2\bigr).
\]
Discarding the absorption term and 
denoting by $M$ the supremum in assumption \eqref{hyp-tails},
we obtain, for $\tau\le t<\Tmax$,
\[
 V_t-\Delta V
 \le  e^{-u}\bigl(u^p+Mu^a+Mu^b\bigr).
\]
Since all three exponents $p,a,b$ are at least $1$,
Lemma~\ref{lem:polyexp} therefore gives
\begin{equation*}
\label{Vlinear}
 V_t-\Delta V\le C_\tau V
 \quad\text{in }[\tau,\Tmax)\times\RN,
 \qquad
 C_\tau:=K_p+ M\bigl(K_a+K_b\bigr).
\end{equation*}
By Lemma~\ref{lem:C0}, $u(\tau)\in C_0(\RN)$, and since $0\le1-e^{-s}\le s$, also $V(\tau)\in C_0(\RN)$.  For every $T\in(\tau,\Tmax)$, Proposition~\ref{cor:linear} on $[\tau,T]$ gives
\begin{equation*}
\label{Vcomp}
 0\le V(t)
 \le e^{C_\tau(t-\tau)}e^{(t-\tau)\Delta}V(\tau)
 \le e^{C_\tau\Tmax}e^{(t-\tau)\Delta}V(\tau),
 \qquad \tau\le t<\Tmax .
\end{equation*}
Lemma~\ref{lem:heat-tail} on $[0,\Tmax-\tau]$ shows that $V$ vanishes at infinity uniformly for $t\in[\tau,\Tmax)$.  
Since $u=-\log(1-V)$ and $-\log(1-s)\to0$ as $s\downarrow0$, 
property \eqref{tails} follows.
\end{proof}

\subsection{Blow-up of gradient ratios at $\Tmax$}

\begin{lemma} 
\label{prop:tails2}
Let $\Omega,p,q,\mu,u_0,\tau$ be as in Lemma~\ref{prop:tails}. Then, for all $\alpha\in(0,1/2]$,
\be{hyp-tails2}
\sup_{(\tau,\Tmax)\times{\R}^N} \,\frac{|\nabla u|^2}{u^2+u^{2-2\alpha}}=\infty.
\ee
\end{lemma}

\begin{proof}
Assume for contradiction that \eqref{hyp-tails2} fails for some $\alpha\in(0,1/2)$.
Consequently $Z:=\arsinh(u^\alpha)$ satisfies the Lipschitz bound
\[
\tilde M:=\sup_{(\tau,\Tmax)\times{\R}^N} \, |\nabla Z|<\infty.
\]
On the other hand, since  \eqref{hyp-tails} is satisfied with $a=2$ and $b=2-2\alpha$, Lemma~\ref{prop:tails} 
guarantees a uniform control at large distance; there exists $R>0$ such that 
\begin{equation}
\label{exterior-bound}
 u(t,x)\le1,
 \qquad \tau\le t<\Tmax,
 \quad|x|\ge R .
\end{equation}
Let $|x|<R$ and choose $x'$ with $|x'|=R$ and $|x-x'|\le R$: take $x'=Rx/|x|$ if $x\ne0$, and $x'=Re_1$ if $x=0$.  Then, by \eqref{exterior-bound},
\[
 Z(t,x)\le Z(t,x')+\tilde M R  \le\arsinh1+\tilde MR,
\]
whence
\[
 \sup_{\tau\le t<\Tmax}\norm{u(t)}_\infty
 \le\Bigl[\sinh\bigl(\arsinh1+\tilde M R\bigr)\Bigr]^{1/\alpha}
 <\infty .
\]
Since $\Tmax<\infty$, this is a contradiction with the blow-up alternative in Proposition~\ref{BUalt}. 
\end{proof}

\subsection{Completion of proof of Theorem~\ref{thm:main-BUC}}

\indent {\bf Case 1:} $\Omega=\R^N$.
 Assume for contradiction that $\Tmax<\infty$.  
Fix some $\alpha\in(0,1/2)$ and $\tau\in(0,\Tmax)$. Theorem~\ref{thm:universal} guarantees 
 that property \eqref{hyp-tails2} fails.
This is a contradiction with Lemma~\ref{prop:tails2}. Consequently  $\Tmax=\infty$.
\smallskip

\noindent {\bf Case 2:} General domain $\Omega$.
Denote by $\bar u_0$ the extension of $u_0$ by $0$ in $\R^N\setminus\Omega$.
Then $\bar u_0\in W^{1,\infty}\cap C_0(\R^N)$. 
By minor modifications of the standard local existence arguments
(cf.,~e.g.,~the proof of Proposition~\ref{lem:C0}), problem \eqref{P} with initial data $\bar u_0$ 
has a unique 
local solution $\bar u$, and it satisfies $\bar u(t)\in BC^1(\R^N)\cap C_0(\R^N)$ for $t>0$ small.
It follows from Case 1 that $\bar u$ exists globally.
Assume for contradiction that $\Tmax<\infty$. 
By the blow-up alternative in Proposition~\ref{BUalt}, we thus have $\lim_{t\in (0,\Tmax)}\|u(t)\|_\infty=\infty$.
On the other hand, by the comparison principle in 
\cite[Proposition 52.10]{QuittnerSouplet2019}, we have $0\le u\le \bar u$ on $(0,\Tmax)\times\Omega$,
hence $\sup_{t\in (0,\Tmax)}\|u(t)\|_\infty\le \sup_{t\in (0,\Tmax)}\|\bar u(t)\|_\infty<\infty$. This is a contradiction.
\qed

\begin{remark}[Exterior localization]
\label{rem:localization}
Instead of the Lipschitz argument in the proof of Lemma~\ref{prop:tails} above, the proof of Theorem~\ref{thm:main-BUC} can be concluded with the barrier of Remark~\ref{rem:barrier}, in the spirit of \cite{SoupletWeissler1999}.  
 Consider the case $\Omega=\RN$ for simplicity. 
Suppose $\Tmax<\infty$ and fix $\tau\in(0,\Tmax)$. By \eqref{exterior-bound} we have
\[
 \sup_{\tau<t<\Tmax, \,|x|\ge R}u(t,x)\le M_0<\infty,
\]
for some $R>0$.
Set
\[
 k=\mu^{-1/q},
 \qquad
 M=\max\{1,M_0,\norm{u_0}_\infty\},
 \qquad
 U(t,x)=M\exp\bigl(k^2t+k(x_1+R)\bigr).
\]
For $|x|<R$ one has $x_1\ge-R$, hence $U\ge M\ge1$, so $U$ is a supersolution in $(0,\Tmax)\times\{|x|<R\}$ by the computation of Remark~\ref{rem:barrier}.  Moreover $U(0,\cdot)\ge M\ge u_0$ on $\{|x|\le R\}$, and $U\ge M\ge M_0\ge u$ on $\{|x|=R\}$ for $0<t<\Tmax$.  Comparison on $(0,T)\times\{|x|<R\}$, $T<\Tmax$, in the classical bounded-domain form 
gives
$u\le M\exp\bigl(k^2\Tmax+2kR\bigr)$
in $\{|x|<R\}$,
so $u$ is bounded on $[0,\Tmax)\times\RN$,  a contradiction as above.  This route uses a bounded-domain comparison instead of the global Lipschitz bound, but it needs the same tail estimate as input.
\end{remark}

\section{Asymptotically constant data}
\label{sec:dichotomy}

\subsection{Preservation of spatial limits}

 The following proposition, which is a variation of the local theory,
 is our basic tool for the study of asymptotically constant data.
It also recovers the (known) fact that the class $C_0$ is preserved by the flow.

\begin{proposition}
\label{lem:C0}
Let $\Omega=\R^N$, $p,q>1$, $\mu\in\R$, 
 $u_0\in X_+$, with $u_0(x)\to\ell\ge0$ as $|x|\to\infty$.  Let $\zeta_\ell$ solve $\zeta'=\zeta^p$, $\zeta(0)=\ell$, on $[0,T_\ell)$, with $\zeta_0\equiv0$ and $T_0:=\infty$.  Then
the solution of \eqref{P} satisfies
\be{uzetaC0}
 u(t)-\zeta_\ell(t)\in C_0(\RN),
 \qquad 0\le t<\min\{\Tmax,T_\ell\}.
\ee
In particular, if $u_0\in C_0(\RN)\cap X_+$, then $u(t)\in C_0(\RN)$ for all $t\in[0,\Tmax)$.
\end{proposition}

 Its proof uses the following elementary fact.
 
\begin{lemma}
\label{lem:C0Lem}
If $g\in \BUC^1(\RN)$ and $g(x)\to\ell\in\R$ as $|x|\to\infty$, then $\lim_{|x|\to\infty}\nabla g(x)=0$.
\end{lemma}

 We recall the simple proof for completeness.

\begin{proof}
Suppose not.  Then there exist $\eps>0$ and $|x_n|\to\infty$ such that $|\nabla g(x_n)|\ge2\eps$.  Put
$e_n:=\frac{\nabla g(x_n)}{|\nabla g(x_n)|}$.
Since $\nabla g$ is uniformly continuous, there exists $r>0$, independent of $n$, such that
$|\nabla g(x_n+se_n)-\nabla g(x_n)|\le\eps$
 for all $s\in[0,r]$, hence $\nabla g(x_n+se_n)\cdot e_n\ge\eps$ for $0\le s\le r$.
 Therefore
\[
 g(x_n+re_n)-g(x_n)
 =\int_0^r\nabla g(x_n+se_n)\cdot e_n\,ds
 \ge\eps r .
\]
Both terms on the left tend to $\ell$, a contradiction.
\end{proof}

\begin{proof}[Proof of Proposition~\ref{lem:C0}]
Set $C_0^1(\RN):=\{f\in C^1(\RN):\ f,\nabla f\in C_0(\RN)\}$ and denote by $e^{t\Delta}$ the heat semigroup on $\RN$.
We will use the well-known semigroup estimates (see, e.g.,~\cite{Pazy1983}):
\begin{equation}
\label{smoothing}
 \norm{e^{t\Delta}f}_\infty\le\norm{f}_\infty,
 \qquad
 \norm{\nabla e^{t\Delta}f}_\infty\le C_Nt^{-1/2}\norm{f}_\infty,
 \qquad f\in L^\infty(\RN),\ t>0,
\end{equation}
 and 
\begin{equation}
\label{smoothing2}
\norm{\nabla e^{t\Delta}f}_\infty\le \norm{\nabla f}_\infty,
 \qquad f\in BC^1(\RN),\ t>0.
\end{equation}
Fix $T<\min\{\Tmax,T_\ell\}$ and write $\zeta=\zeta_\ell$ and $w:=u-\zeta$ on $[0,T]$.  Since $e^{t\Delta}$ fixes constants and $\zeta(t)=\ell+\int_0^t\zeta(s)^p\,ds$, subtracting from \eqref{mild} gives
\[
 w(t)= [\Phi(w)](t):=e^{t\Delta}w(0)
 +\int_0^te^{(t-s)\Delta}F(s,w(s))\,ds,
 \quad
 F(s,v):=\bigl(\zeta(s)+v\bigr)^p-\zeta(s)^p-\mu|\nabla v|^q. 
\]  
 For $t>0$,
$e^{t\Delta}$ maps $BC(\RN)$ into $BC^2(\RN)$
(with $\|e^{t\Delta}\phi\|_{BC^2}\le Ct^{-1}\|\phi\|_\infty$).
Since $BC^2(\RN)\subset BUC^1(\RN)$,
it follows from Lemma~\ref{lem:C0Lem} that 
\be{mapsC0C1}
\hbox{for each $t>0$, $e^{t\Delta}$ maps $C_0(\RN)$ into $C_0^1(\RN)$.}
\ee
Set $W(t)=e^{t\Delta}w(0)$ and $M=\|w(0)\|_{C^1}$. Let $\tau\in (0,\min(1,T)]$ to be chosen below and
put $K_\tau=\sup_{t\in [0,\tau]}\zeta(t)$.
Define
$$E=L^\infty(0,\tau;C^1_0(\RN)),\quad\hbox{with norm } \|v\|_E={\rm ess}\sup_{t\in (0,\tau)}\|v(t)\|_{C^1},
\qquad Y:=\bar B_E(0,2M).$$
Then $Y$ is a complete metric subspace of $E$.
By \eqref{smoothing}--\eqref{mapsC0C1},
since $w(0)\in C_0\cap BC^1(\RN)$, we have $W\in E$ with $\|W(t)\|_E\le M$.
Denote by $\mathcal{B}$ the closed ball of radius $2M$ in $C^1_0(\RN)$.
There exists a constant $L=L(M,K_T,p,q,\mu)>0$ such that, for each $s\in [0,T]$, the map $F(s,\cdot)$ 
 is $L$-Lipschitz from $\mathcal{B}$ to $C_0(\RN)$.
Let $w,\tilde w\in Y$. 
By \eqref{smoothing} it follows that, for all $t\in(0,\tau]$,
$$
\begin{aligned}
\int_0^t \|e^{(t-s)\Delta}&[F(s,w(s))-F(s,\tilde w(s))]\|_{C^1}\,ds \\
&\le C\int_0^t (t-s)^{-1/2} \|F(s,w(s))-F(s,\tilde w(s))\|_\infty ds\le CL\tau^{1/2}\|w-\tilde w\|_E,
\end{aligned}
$$
 where $C=C(N)>0$.
 In particular, taking $\tilde w=0$ and using $F(s,0)=0$, 
 we deduce that $\Phi(w)\in E$ with
$$
 \|\Phi(w)\|_E\
\le \|W\|_E+CL\tau^{1/2}\|w\|_E\le M(1+2CL\tau^{1/2}),$$
and then that
$$
 \|\Phi(w)-\Phi(\tilde w)\|_E
\le CL\tau^{1/2}\|w-\tilde w\|_E.$$
Taking $\tau=\min(T,(1+2CL)^{-2})$, so that $CL\tau^{1/2}\le 1/2$,  
we deduce that the operator $\Phi$ is a contraction from $Y$ to itself,
hence admits a unique fixed point.
By uniqueness in  $BC^1(\RN)$, this fixed point coincides with $w$.  Since $\sup_{[0,T]}\norm{w(t)}_{C^1}<\infty$, finitely many restarts give $w\in L^\infty(0,T;C_0(\RN))$,
hence $w(t)\in C_0(\RN)$ for a.e.~$t\in(0,T)$.
 Since we already know that $u\in C([0,T];BC(\RN))$, property \eqref{uzetaC0} follows. 
\end{proof}

\subsection{Proof of Theorem~\ref{thm:asymptotically-constant}}
 Let
\[
 \zeta(t):=\bigl(\ell^{1-p}-(p-1)t\bigr)^{-1/(p-1)},
 \qquad 0\le t<T_\ell,
\]
be the solution of $\zeta'=\zeta^p$, $\zeta(0)=\ell$, so that $\zeta(t)\to\infty$ as $t\uparrow T_\ell$.
By  Proposition~\ref{lem:C0} we have
$u(t)-\zeta(t)\in C_0(\RN)$, hence
\be{evolving-tail2}
\lim_{|x|\to\infty} u(t,x)=\zeta(t),
 \qquad 0<t<\min\{\Tmax,T_\ell\}.
\ee
In particular
$\norm{u(t)}_\infty\ge\zeta(t)$ for all $t\in (0,\min\{\Tmax,T_\ell\})$.
 Since $\lim_{t\to T_\ell} \zeta(t)=\infty$, it follows that
$\Tmax\le T_\ell$, and \eqref{evolving-tail2} implies \eqref{evolving-tail}.
\medskip

 Next set $m:=\norm{u_0}_\infty\ge\ell$ and let
\[
 \psi(t):=\bigl(m^{1-p}-(p-1)t\bigr)^{-1/(p-1)},
 \qquad 0\le t<T_m,
\]
which solves $\psi'=\psi^p$, $\psi(0)=m$.  
 Assume for contradiction that $\Tmax<T_m$. By the comparison principle 
(cf.~Proposition~\ref{cor:comparison}),
we have $u\le \psi$ on $[0,\Tmax)\times\RN$, hence
$
 \sup_{0\le t<\Tmax}\norm{u(t)}_\infty\le \psi(\Tmax)<\infty,
$
which contradicts the blow-up alternative in Proposition~\ref{BUalt}. 
We conclude that $T_m\le\Tmax\le T_\ell$.
In particular, $T_m=T_\ell$ when $m=\ell$.
\qed

\begin{remark} \label{JustifClaim}
Let us justify the claim in Remark~\ref{AsymptConst}(c).
We first note that there exists $v_0\in X_+$ such that 
\be{propv0u0}
0\le v_0\le u_0\quad\hbox{and}\quad \lim_{|x|\to\infty}v_0(x)=\ell.
\ee
Indeed, since $\ell:=\liminf_{|x|\to\infty}u_0(x)>0$, there exists a sequence $R_n\to \infty$ with $R_0=0$
such that, for each $n\ge 1$, $u_0\ge (1-2^{-n-1})\ell$ on $\{|x|\ge R_n\}$
and $R_n-R_{n-1}\ge 1$.
We may then find a nondecreasing radially symmetric function $v_0\in X_+$ such that
$v_0(R_n)=(1-2^{-n})\ell$ for all $n\ge 0$, hence 
$v_0\le (1-2^{-n-1})\ell\le u_0$
on the annulus $\{R_n\le |x|\le R_{n+1}\}$.
This implies \eqref{propv0u0}.
Now we have $\Tmax(v_0)\le T_\ell<\infty$ by Theorem~\ref{thm:asymptotically-constant}, and 
 $\lim_{t\uparrow\Tmax}\norm{v(t)}_{\infty}=\infty$
 by the blow-up alternative Proposition~\ref{BUalt}. 
 Since, by the comparison principle, $u\ge v$ on the common existence interval,
it follows that $\Tmax(u_0)\le \Tmax(v_0)\le T_\ell$, whereas $\Tmax(u_0)\ge T_{\norm{u_0}_\infty}$
by the above proof.
\end{remark}

\appendix

\section{Auxiliary results: Local theory and comparison principles}
\label{sec:prelim}

\subsection{Local classical theory} 
\label{LocTheory}

We start by recalling the standard local theory in $BC^1$.
Here we denote by $S_D(t)$ the Dirichlet heat semigroup on $\Omega$
(or the Gaussian semigroup $e^{t\Delta}$ in case $\Omega=\RN$).

\begin{proposition}
\label{prop:local}
Assume \eqref{hypOmega}, let $p,q>1$, $\mu\in \R$ and $u_0\in  X_+$.
\begin{enumerate}[label=\textup{(\roman*)},leftmargin=*]
\item There exist $\Tmax\in(0,\infty]$ and a unique $u\in C([0,\Tmax);X_+)$ with
\begin{equation}
\label{mild}
 u(t)=S_D(t)u_0
 +\int_0^t S_D(t-s)[u^p(s)-\mu|\nabla u(s)|^q]\,ds,
 \qquad 0\le t<\Tmax,
\end{equation}
which cannot be continued in this class beyond $\Tmax$.  
If $u_0\not\equiv0$, then $u>0$ in $(0,\Tmax)\times\Omega$.
If $\Tmax<\infty$, then
\be{blowupalt}
\lim_{t\uparrow\Tmax}\norm{u(t)}_{C^1}=\infty.
\ee
\item For $0<t_0<T<\Tmax$, the functions $u,\nabla u,D^2u,u_t$ are bounded and H\"older continuous on $[t_0,T]\times\overline\Omega$, and $u$ is a classical solution of $u_t-\Delta u=u^p-\mu|\nabla u|^q$ in $(0,\Tmax)\times\Omega$.  
 Moreover, 
\be{HigherReg}
  \nabla u\in C^{2,1}\bigl((0,\Tmax)\times\Omega\bigr).
\ee
\item If $u_{0,n}\to u_0$ in $X_+$ and $T<\Tmax(u_0)$, then $\Tmax(u_{0,n})>T$ for $n$ large, and the corresponding solutions satisfy $u_n\to u$ in $C([0,T];X_+)$.
\end{enumerate}
\end{proposition}

 See, e.g., \cite[Remark 51.11 p.584]{QuittnerSouplet2019}
(and cf.~also \cite{Lun}).
As for the higher regularity \eqref{HigherReg}, it follows from a standard use of Schauder parabolic 
estimates as in, e.g., \cite[p.238 and Remark~48.3(i)]{QuittnerSouplet2019}.

\smallskip

When $\mu\ge 0$, the blow-up alternative can be stengthened
as follows (see \cite{SoupletWeissler1999,BartierSouplet2004} and 
\cite[Theorem 35.1]{QuittnerSouplet2019}).

\begin{proposition}
\label{BUalt}
Under the assumptions of Proposition~\ref{prop:local},
if $\mu\ge 0$ and $\Tmax<\infty$, then
\be{blowupalt2}
 \lim_{t\uparrow\Tmax}\norm{u(t)}_{\infty}=\infty.
\ee
\end{proposition}

\subsection{Comparison principles}

The comparison arguments in this paper reduce to the following lemma.
The penalization by $(1+|x|^2)^{1/2}$ is the standard whole-space device; see \cite{LSU1968,QuittnerSouplet2019}.
We give proofs for the convenience of readers.

\begin{lemma}[Comparison with a spatially homogeneous barrier]
\label{lem:MP}
Let $0\le s<T$.  Let $W\in C([s,T]\times\RN)\cap C^{2,1}((s,T]\times\RN)$ be bounded, let $b:(s,T]\times\RN\to\RN$ be bounded, and let $h\in C([s,T])\cap C^1((s,T])$.  Assume $W(s,\cdot)\le h(s)$ on $\RN$, and that there is $c>0$ such that
\begin{equation}
\label{MP-hyp}
 W_t-\Delta W-b\cdot\nabla W\le h'(t)-c
 \quad\text{at every }(t,x)\in(s,T]\times\RN
 \text{ with }W(t,x)>h(t).
\end{equation}
Then $W\le h$ in $[s,T]\times\RN$.
\end{lemma}

\begin{proof}
Let $b_\infty:=\sup|b|$ and $\rho(x):=(1+|x|^2)^{1/2}$, so that $|\nabla\rho|\le1$ and
$
 \Delta\rho=\frac N\rho-\frac{|x|^2}{\rho^3}\le N .
$
Fix $\kappa\in\bigl(0,c/(N+b_\infty)\bigr)$ and set
\[
 \Psi:=W-h-\kappa\rho .
\]
Let $C_1:=\sup|W|+\max_{[s,T]}|h|$, and let $\mathcal K_\kappa:=\{x:\kappa\rho(x)\le C_1+1\}$, a compact set.  Then $\Psi<-1$ outside $[s,T]\times\mathcal K_\kappa$, so any nonnegative value of $\Psi$ is attained on the compact set $\mathcal K_\kappa$.

The function $m(t):=\max_{x\in\mathcal K_\kappa}\Psi(t,x)$ is continuous on $[s,T]$, and $m(s)\le-\kappa<0$.  Suppose that $\Psi>0$ somewhere.  Then
$
 t_1:=\min\{t\in[s,T]:\ m(t)\ge0\}\in(s,T]
$
is well defined, $m(t_1)=0$, and $m<0$ on $[s,t_1)$.  Pick $x_1\in\mathcal K_\kappa$ with $\Psi(t_1,x_1)=0$.

Then $x_1$ is a global maximum point of $\Psi(t_1,\cdot)$, so $\nabla\Psi=0$ and $\Delta\Psi\le0$ at $(t_1,x_1)$; moreover $\Psi(t,x_1)<0$ for $t<t_1$, whence $\Psi_t(t_1,x_1)\ge0$.  At $(t_1,x_1)$ we also have $W=h+\kappa\rho>h$, so \eqref{MP-hyp} applies and gives
\[
 h'-c
 \ \ge\ W_t-\Delta W-b\cdot\nabla W
 =(h'+\Psi_t)-(\Delta\Psi+\kappa\Delta\rho)
 -b\cdot(\nabla\Psi+\kappa\nabla\rho)
 \ \ge\ h'-\kappa(N+b_\infty),
\]
that is, $c\le\kappa(N+b_\infty)$, a contradiction.  Thus $\Psi\le0$, and letting $\kappa\downarrow0$ gives $W\le h$.
\end{proof}

For $T>0$ we write
\[
 \mathcal S_T:=C([0,T];BC^1(\RN))\cap C^{2,1}((0,T]\times\RN),
 \qquad
 P[v]:=v_t-\Delta v-v_+^p+\mu|\nabla v|^q .
\]

\begin{proposition}[Comparison principle]
\label{cor:comparison}
Let $u_1,u_2\in\mathcal S_T$ satisfy  $P[u_1] 
\le P[u_2]$ in $(0,T]\times\RN$ and $u_1(0)\le u_2(0)$.  Then $u_1\le u_2$ in $[0,T]\times\RN$.
\end{proposition}

\begin{proof}
Let $w:=u_1-u_2$.  By the mean value theorem applied to $s\mapsto s_+^p$ and $\xi\mapsto|\xi|^q$,
we have $w_t-\Delta w\le c_1w+\beta\cdot\nabla w$, 
where
\[
 c_1:=\int_0^1p\bigl(\lambda u_1+(1-\lambda)u_2\bigr)_+^{p-1}\,d\lambda,
 \quad\ 
 \beta:=-\mu q\int_0^1|\xi_\lambda|^{q-2}\xi_\lambda\,d\lambda,
 \quad\ 
 \xi_\lambda:=\lambda\nabla u_1+(1-\lambda)\nabla u_2 .
\]
Both $c_1, \beta$ are bounded; let $\Lambda\ge\sup c_1$ and $\varphi:=e^{-\Lambda t}w$, so that
$
 \varphi_t-\Delta\varphi-\beta\cdot\nabla\varphi\le(c_1-\Lambda)\varphi .
$
Fix $\eps>0$ and set $h(t):=\eps(1+t)$.  At points where $\varphi>h>0$ the right-hand side is $\le0=h'-\eps$.  Since $\varphi(0)\le0<h(0)$, Lemma~\ref{lem:MP}, with margin $\eps$, gives $\varphi\le\eps(1+t)$.  Letting $\eps\downarrow0$ gives $w\le0$.
\end{proof}

\begin{proposition}[Linear comparison]
\label{cor:linear}
Let $0\le s<T$ and $C\ge0$.  Let $V\in C([s,T]\times\RN)\cap C^{2,1}((s,T]\times\RN)$ be bounded, with $V(s,\cdot)\in \BUC(\RN)$ and
\[
 V_t-\Delta V\le CV
 \qquad\text{in }(s,T]\times\RN .
\]
Then
\[
 V(t)\le e^{C(t-s)}e^{(t-s)\Delta}V(s),
 \qquad s\le t\le T .
\]
\end{proposition}

\begin{proof}
Let $\varphi(t):=e^{-C(t-s)}V(t)-e^{(t-s)\Delta}V(s)$.  Since $V(s)\in \BUC(\RN)$, $\varphi$ is bounded and lies in $C([s,T]\times\RN)\cap C^{2,1}((s,T]\times\RN)$; moreover $\varphi(s)=0$ and
\[
 \varphi_t-\Delta\varphi
 =e^{-C(t-s)}\bigl(V_t-\Delta V-CV\bigr)\le0 .
\]
Lemma~\ref{lem:MP} with $b=0$, $h(t)=\eps(1+t-s)$ and margin $\eps$ gives $\varphi\le\eps(1+t-s)$ for every $\eps>0$.
\end{proof}

\begin{proposition}  \emph{(Positive floor.)} 
\label{cor:consequences}
Let $u_0\in X_+$, and let $u$ be given by Proposition~\ref{prop:local}.
Let $T<\Tmax$ and, for $\eta>0$, let $u_\eta$ denote the solution with initial datum $u_0+\eta$.  There is $\eta_0>0$ such that, for $0<\eta\le\eta_0$,
\[
 \Tmax(u_0+\eta)>T,
 \qquad
 u_\eta\ge\eta\ \text{ in }[0,T]\times\RN,
 \qquad
 u_\eta\to u\ \text{ in }C([0,T];BC^1(\RN))
 \ \text{ as }\eta\downarrow0 .
\]
\end{proposition}

\begin{proof}
The first and third assertions follow from Proposition~\ref{prop:local}(iii).  The constant function $\eta$ belongs to $\mathcal S_T$ and satisfies $P[\eta]=-\eta^p\le0$, so Proposition~\ref{cor:comparison} gives $\eta\le u_\eta$.
\end{proof}

{\bf Declarations.} The authors report that there are no competing interests to declare. This manuscript has no associated data.

\smallskip

{\bf Declaration on the use of generative AI.}
The authors used OpenAI's ChatGPT for language editing, manuscript presentation, and checking mathematical arguments. All AI-assisted content was independently verified, and the authors take full responsibility for the manuscript.

\small

\noindent
\textsc{Philippe Souplet}\\
Universit\'e Sorbonne Paris Nord, CNRS UMR 7539, \\Laboratoire Analyse, G\'eom\'etrie et 
Applications, 93430 Villetaneuse, France.
\\ E-mail: \texttt{souplet@math.univ-paris13.fr}\\{}\\
\textsc{Berikbol T. Torebek}\\
Institute of
Mathematics and Mathematical Modeling
\\Shevchenko street 28, 050010 Almaty, Kazakhstan\\ and \\
SDU University \\ Abylai Khan street 1/1, 040900 Kaskelen, Kazakhstan.
\\ E-mail: \texttt{torebek@math.kz}
\end{document}